\documentclass[letterpaper, 10 pt, conference]{ieeeconf}

\IEEEoverridecommandlockouts

\let\labelindent\relax

\usepackage{cite}
\usepackage[pdftex]{graphicx}
\usepackage[cmex10]{amsmath}
\usepackage{amsthm}
\usepackage{amssymb}
\usepackage{enumitem}
\usepackage{savesym}
\savesymbol{AND}
\usepackage{algorithmic}
\usepackage{algorithm}
\usepackage{tikz}
\usepackage{caption}
\usepackage{subcaption}
\usepackage{array}
\usepackage{booktabs}
\usepackage{rotating}
\usepackage{minibox}

\usetikzlibrary{decorations.pathreplacing,shapes,arrows,calc}

\tikzstyle{block} = [draw, fill=white!20, rectangle, 
    minimum height=1cm, minimum width=1cm,rounded corners=5pt]
\tikzstyle{sum} = [draw, fill=white!20, circle, node distance=1cm]
\tikzstyle{input} = [coordinate]
\tikzstyle{output} = [coordinate]
\tikzstyle{tmp} = [coordinate]
\tikzstyle{pinstyle} = [pin edge={<-,line width=0.3mm,black}]

\theoremstyle{plain}
\newtheorem{thm}{Theorem}

\newtheorem{lem}{Lemma}

\theoremstyle{definition}
\newtheorem{defn}{Definition}

\newtheorem{assum}{Assumption}

\begin{document}

\title{\LARGE \bf 
  Performance of Model Predictive Control of POMDPs}

\author{
  Martin A. Sehr \& Robert R. Bitmead
  \thanks{The authors are with the Department
    of Mechanical and Aerospace Engineering, UC San Diego, La
    Jolla, CA 92093, USA. 
\newline{\tt\small \{msehr,rbitmead\}@ucsd.edu}}%
}

\maketitle

\begin{abstract}
We revisit closed-loop performance guarantees for Model Predictive Control in the deterministic and stochastic cases, which extend to novel performance results applicable to receding horizon control of Partially Observable Markov Decision Processes. While performance guarantees similar to those achievable in deterministic Model Predictive Control can be obtained even in the stochastic case, the presumed stochastic optimal control law is intractable to obtain in practice. However, this intractability relaxes for a particular instance of stochastic systems, namely Partially Observable Markov Decision Processes, provided reasonable problem dimensions are taken. This motivates extending available performance guarantees to this particular class of systems, which may also be used to approximate general nonlinear dynamics via gridding of state, observation, and control spaces. We demonstrate applicability of the novel closed-loop performance results on a particular example in healthcare decision making, which relies explicitly on the duality of the control decisions associated with Stochastic Optimal Control in weighing appropriate appointment times, diagnostic tests, and medical intervention for treatment of a disease modeled by a Markov Chain.
\end{abstract}

\section{Introduction}
\label{intro}
Model Predictive Control (MPC) is well applied and popular because of its capacity to handle constraints and its simple formulation as an open-loop finite-horizon optimization problem evaluated on the receding horizon~\cite{mayne2000constrained,mayne2014model}. There are a few areas in which MPC is wanting for more complete results, notably in the area of output feedback control and the associated requirement to manage the duality of the control signal in stochastic MPC (SMPC) problems. When SMPC is developed as a logical extension of finite-horizon Stochastic Optimal Control, which demands computation of closed-loop policies, it inherits the computational intractability of this latter subject via the inclusion of the Bayesian filter, required to propagate the conditional state densities, and the stochastic dynamic programming equation. 

Results exist relating the infinite-horizon performance of MPC to both the optimal performance and the performance computed as part of the finite-horizon optimization. These performance bounds are available in both the deterministic \cite{grune2008infinite} and the stochastic \cite{sehr2016stochastic} settings, were one ever able to solve the underlying finite-horizon stochastic problem computationally. While approximation of SMPC based on Stochastic Optimal Control via more tractable surrogate problems is possible, such as for instance in~\cite{sui2008robust ,mayne2009robust,blackmore2010probabilistic,sehr2017particle}, one generally loses the associated closed-loop guarantees, in particular regarding infinite-horizon performance of the generated control laws. 

In this paper, we derive new performance results for SMPC of systems described by Partially Observable Markov Decision Processes (POMDPs, see e.g.~\cite{smallwood1973optimal,kaelbling1998planning}). POMDP system models of small to moderate dimensions admit tractable computation of finite-horizon stochastic optimal control laws while preserving the control signal duality, and so are attractive propositions with which to approach implementable SMPC~\cite{sunberg2013information,sehr2017tractable}. In deriving perfomance bounds for this specific class of problems, we examine their relation to the deterministic and stochastic continuous-state results, highlighting the role of value function monotonicity with horizon.
All theorems discussed in this paper exhibit the same conceptual structure:
\begin{center}
\minibox[frame,c]{\emph{Assumption}: \\ Terminal cost contraction}\\
$\Downarrow$\\
\minibox[frame,c]{\emph{Observation}: \\ Value function monotonicity with horizon}\\
$\Downarrow$\\
\minibox[frame,c]{\emph{Result}:\\
Infinite-horizon optimal cost\\
\begin{turn}{90} $\geq$ \end{turn}\\
Achieved infinite-horizon MPC cost\\
\begin{turn}{90} $\geq$ \end{turn}\\
Computed $N$-horizon optimal cost}
\end{center}

While the capability of handling constraints is a raison-d'\^etre for MPC, constraints complicate this analysis and add little to the discussion about closed-loop cost. Thus, as in most of~\cite{grune2008infinite}, we omit the explicit consideration in this paper and point out that constraints may be reinserted subject to recursive feasibility assumptions. 

The paper is organized as follows. We revisit a particular infinite-horizon performance result from~\cite{grune2008infinite} in Section~\ref{sec:det}. We then proceed by reviewing a stochastic counterpart to this result, derived in~\cite{sehr2016stochastic}, which we extend to receding horizon control of POMDPs in Section~\ref{sec:pomdp}. A specific POMDP example from healthcare is studied in Section~\ref{sec:eg} to demonstrate numerically the satisfaction of assumptions, interpret control duality, and evaluate performance bounds on the infinite control horizon. The example, introduced in~\cite{sehr2017tractable}, displays in particular the dual nature of SMPC based on Stochastic Optimal Control. 

\section{Deterministic Model Predictive Control}
\label{sec:det}
This section revisits a performance result for deterministic MPC from~\cite{grune2008infinite}, which we extend to SMPC for nonlinear systems (see also~\cite{sehr2016stochastic}) and POMDPs below. Consider the nonlinear dynamic system
\begin{align*}
x_{t+1} &= f(x_t,u_t),
\end{align*}
where $x_t\in X$ and $u_t\in U$ for $t\in\mathbb{N}_0\triangleq\{0,1,2,\ldots\}$ and metric spaces $X, U$. Further define the space of control sequences $u:\mathbb{N}_0\to U$ as $\mathcal{U}$. In principle, we aim to find control policy $\mu:X\to U$ that minimizes the infinite-horizon cost functional
\begin{align}\label{eq:detcostih}
J_{\infty}(x_0,u) \triangleq \sum_{k = 0}^{\infty} c(x_k,u_k),
\end{align}
where $c:X\times U\to\mathbb{R}_+$ is the stage cost. We define the optimal value function associated with cost~\eqref{eq:detcostih} as
\begin{align*}
J_{\infty}^{\star}(x_0)\triangleq\inf_{u} J_{\infty}(x_0,u)
\end{align*}
Given that solution of this infinite-horizon optimal control problem, even in the deterministic case, is usually intractable, a popular approach is to replace~\eqref{eq:detcostih} by a finite-horizon optimal control problem over horizon $N\in\mathbb{N}_0$, with cost functional
\begin{align}\label{eq:detcostfh}
J_{N}(x_0,u) \triangleq \sum_{k = 0}^{N-1} c(x_k,u_k) + c_N(x_N),
\end{align}
where $c_N:X\to\mathbb{R}_+$ denotes an optional terminal cost term. The optimal value function corresponding to~\eqref{eq:detcostfh} is defined by
\begin{align}\label{eq:detvaluefh}
J_{N}^{\star}(x_0)\triangleq\inf_{u} J_{N}(x_0,u).
\end{align}
We further denote the sequence of optimal control policies in this finite-horizon problem by $\mu^N$, with first control policy $\mu_0^N:X\to U$, which is implemented repeatedly in an MPC law, denoted by 
\begin{align*}
\mu_{\text{MPC}}^{N} \triangleq \{\mu_0^N,\mu_0^N,\ldots\}. 
\end{align*}
We now aim to provide computational estimates of the infinite-horizon achieved MPC cost $J_{\infty}(x_0,\mu_{\text{MPC}}^{N})$ in relation to the computed finite-horizon optimal cost $J_{N}(x,\mu^{N})$. This goal can be achieved, for instance, by using the following assumption.
\begin{assum}\label{assm:det}
For all $x\in X$, there exists $u\in U$ such that
\begin{align*}
c_N(f(x,u)) &\leq c_N(x) - c(x,u).
\end{align*}
\end{assum}
This assumption on the terminal cost $c_N$ in~\eqref{eq:detcostfh} then leads to the following performance guarantee.
\begin{thm}[Performance of deterministic MPC~\cite{grune2008infinite}]\label{thm:det}
\hspace{2cm}
Given Assumption~\ref{assm:det}, the inequality
\begin{align*}
J_{\infty}^{\star}(x) \leq J_{\infty}(x,\mu_{\text{MPC}}^{N}) \leq J_N^{\star}(x)
\end{align*}
holds for all $x\in X$.
\end{thm}
This result, which is a special case of Theorem~6.2 in~\cite{grune2008infinite}, allows us to provide bounds on the achieved infinite-horizon performance of the closed-loop system when choosing the terminal cost, $c_N$, as a Lyapunov function. This result is particularly useful because we compute the upper bound implicitly when generating our MPC control law, $\mu^{N}_{\text{MPC}}$. Theorem~\ref{thm:det} follows given that Assumption~\ref{assm:det} implies that the underlying finite-horizon optimal value function $J_{N}^{\star}(x)$ is monotonically non-increasing with increasing control horizon $N$. Notice, further, how this result not only provides infinite- but also finite-horizon closed-loop performance guarantees. This follows simply by
\begin{align*}
J_{\infty}^M(x,\mu_{\text{MPC}}^{N}) \leq J_{\infty}(x,\mu_{\text{MPC}}^{N}),
\end{align*}
for all $M\in\mathbb{N}_0$ and $x\in X$, where 
\begin{align*}
J_{\infty}^M(x_0,u) \triangleq \sum_{k=0}^M c(x_k,u_k) .
\end{align*} 
The stochastic extension of this observation is of interest in particular for applications such as the healthcare example provided in Section~\ref{sec:eg} below, where infinite-horizon performance may not be of particular interest given the inherent finite-horizon nature of the control problem. We next provide results of similar quality to Theorem~\ref{thm:det} for SMPC and in particular SMPC applied to POMDPs in the following sections.

\section{Stochastic Model Predictive Control}
\label{sec:smpc}
We next discuss closed-loop performance of SMPC as in~\cite{sehr2016stochastic}. Committing a slight abuse of notation, we shall recycle most of the symbols used previously in Section~\ref{sec:det} above. Consider nonlinear stochastic systems of the form
\begin{align}
x_{t+1}&=f(x_t,u_t,w_t), \label{eq:state}\\
y_t&=h(x_t,v_t), \label{eq:output}
\end{align}
where $x_t\in X$, $u_t\in U$, $y_t\in Y$ for $t\in\mathbb{N}_0$ and metric spaces $X, U, Y$, respectively. Starting from known initial state density $\pi_{0|-1} = \operatorname{pdf}(x_0)$, we denote the data available at time $t$ by 
\begin{align*}
\mathbf{\zeta}^t&\triangleq\{y_0,u_0,y_1,u_1,\dots,u_{t-1},y_t\},&
\mathbf{\zeta}^0&\triangleq\{y_0\}.
\end{align*}
We further impose the following standing assumption on the random variables and control inputs.
\begin{assum}\label{assm:sys}
The signals in~(\ref{eq:state}-\ref{eq:output}) satisfy:
\begin{enumerate}[label=\arabic*.]
\item $w_t$ and $v_t$ are i.i.d. sequences with known densities.
\item $x_0, w_t, v_l$ are mutually independent for all $t,l\in\mathbb{N}_0$.
\item The control input $u_t$ at time instant $t\in\mathbb{N}_0$ is a function of the data $\mathbf{\zeta}^t$ and given initial state density $\pi_{0\mid -1}$.
\end{enumerate}
\end{assum}
The \textit{information state,} denoted $\pi_t$, is the conditional probability density function of state $x_t$ given data $\mathbf{\zeta}^t$,
\begin{align*}
\pi_{t}&\triangleq\operatorname{pdf}\left(x_{t}\mid \mathbf{\zeta}^t \right).
\end{align*}
As a result of the Markovian dynamics~(\ref{eq:state}-\ref{eq:output}), optimal control inputs must inherently be \textit{separated} feedback policies (e.g.~\cite{bertsekas1995dynamic,BKKUM1986}). That is, optimal control input $u_{t}$ depends on the data $\mathbf{\zeta}^t$ and initial density $\pi_{0\mid -1}$ solely through the current information state, $\pi_{t}$. Optimality thus requires propagating $\pi_{t}$ and policies $g_t$, where
\begin{align*}
u_t = g_t(\pi_{t}).
\end{align*}
\begin{defn}
$\mathbb{E}_t[\,\cdot\,]$ and $\mathbb{P}_t[\,\cdot\,]$ are expected value and probability with respect to state $x_t$ -- with conditional density $\pi_t$ -- and i.i.d. random variables $\{(w_k,v_{k+1}):k\geq t\}$.
\end{defn}
Notice that stochastic optimal control on the infinite horizon (see~\cite{bertsekas1995dynamic,bertsekas1978stochastic}) typically requires a discount factor $\alpha < 1$, casting the stochastic version of~\eqref{eq:detcostih} as
\begin{align}\label{eq:scostih}
J_\infty(\pi_{0},g)&\triangleq
\mathbb{E}_0\left[\sum_{k=0}^\infty{\alpha^k c(x_k,g_k(\pi_{k}))}\right],
\end{align}
with corresponding finite-horizon cost
\begin{multline}\label{eq:scostfh}
J_N(\pi_{0},g)\triangleq \\
\mathbb{E}_0\left[\sum_{k=0}^{N-1}{\alpha^k c(x_k,g_k(\pi_{k}))} + \alpha^{N}c_{N}(x_{N})\right].
\end{multline}
Defining the optimal value function $J_{N}^{\star}(\pi_0)$ as in~\eqref{eq:detvaluefh}, 
\begin{align*}
J_{N}^{\star}(\pi_0) \triangleq \inf_{g_k(\cdot)} J_N(\pi_0,g) ,
\end{align*}
finite-horizon stochastic optimal feedback policies may be computed, in principle, by solving the stochastic dynamic programming equation,
\begin{multline}\label{eq:DP1}
J_{N-k}^{\star}(\pi_{k}) \triangleq \\
\inf_{g_k(\cdot)}\ \mathbb{E}_k \left[ c(x_k,g_k(\pi_k)) +\alpha J_{N-k-1}^{\star}(\pi_{k+1})\right],
\end{multline}
for $k = 0,\ldots,N-1$. The equation is solved backwards in time, from its terminal value,
\begin{align}\label{eq:DP2}
J_{0}^{\star}(\pi_{N}) &\triangleq \mathbb{E}_N \left[ c_N(x_{N})\right].
\end{align}
Similarly to Section~\ref{sec:det}, we denote by: $J^{\star}_{\infty}(\pi)$ the infinite-horizon optimal value function; $\mu^N$ the sequence of optimal policies in~(\ref{eq:DP1}-\ref{eq:DP2}); $\mu^N_0$ the first element of this sequence; $\mu^N_{\text{MPC}} \triangleq \{\mu^N_0,\mu^N_0,\ldots \}$ the receding horizon implementation of this sequence. We next impose the following stochastic counterpart to Assumption~\ref{assm:det} to discuss the infinite horizon cost of the SMPC law $\mu^N_{\text{MPC}}$.
\begin{assum}\label{assm:smpc}
For $\alpha\in[0,1)$, there exist $\eta\in\mathbb{R}_+$ and a policy $\tilde{g}(\cdot)$ such that
\begin{multline*}
\mathbb{E}_{\pi}\left[\alpha\, c_N(f(x,\tilde{g}(\pi),w)) \right] \stackrel{}{\leq}\\
\mathbb{E}_{\pi}\left[c_N(x) - c(x,\tilde{g}(\pi))\right]
 + \frac{\eta}{\alpha^{N-1}},
\end{multline*}
for all densities $\pi$ of $x\in X$. The expectation $\mathbb{E}_{\pi}[\cdot]$ is with respect to state $x$ -- with conditional density $\pi$ -- and $w$.
\end{assum}
This assumption then leads to the following extension of Theorem~\ref{thm:det} to SMPC of system~(\ref{eq:state}-\ref{eq:output}).
\begin{thm}[Performance of stochastic MPC~\cite{sehr2016stochastic}]\label{thm:smpc}
\hspace{1cm}
Given Assumption~\ref{assm:smpc}, SMPC with $\alpha\in[0,1)$ yields 
\begin{align*}
J_{\infty}^{\star}(\pi) \leq 
J_{\infty}(\pi,\mu^N_{\text{MPC}}) \leq 
J_{N}^{\star}(\pi) + \frac{\alpha}{1 - \alpha}\eta,
\end{align*}
for all densities $\pi$ of $x\in X$.
\end{thm}
This result relates the following quantities in SMPC: \textit{design cost}, $J_N^{\star}(\pi)$, which is evaluated as part of the SMPC computation; \textit{optimal cost}, $J_{\infty}^{\star}(\pi)$, which is unknown (otherwise we would use the infinite-horizon optimal policy); and, unknown infinite-horizon SMPC \textit{achieved cost} $J_\infty(\pi,\mu^N_{\text{MPC}})$. The result, which must exhibit duality and satisfaction of the stochastic programming equation~(\ref{eq:DP1}-\ref{eq:DP2}), is special in that SMPC approaches relying on approximation of the finite horizon Stochastic Optimal Control problem, as commonly found in the literature, do not generally yield statements regarding performance of the implemented control laws on the infinite horizon. This fact is linked inherently to the loss of the dual optimal nature of the control inputs when avoiding solution of~(\ref{eq:DP1}-\ref{eq:DP2}). 

As in Section~\ref{sec:det} and Theorem~\ref{thm:det}, the proof of Theorem~\ref{thm:smpc} via Assumption~\ref{assm:smpc} relies on verifying monotonicity of the underlying optimal value function $J_N^{\star}(\pi)$. We next proceed by extending this result and its proof to dual optimal receding horizon control of POMDPs.

\section{Stochastic MPC for POMDPs}
\label{sec:pomdp}
POMDPs are characterized by probabilistic dynamics on a finite state space $X = \{1,\ldots,n\}$, finite action space $U = \{1,\ldots,m\}$, and finite observation space $Y = \{1,\ldots,o\}$. POMDP dynamics are defined by the conditional state transition and observation probabilities
\begin{align}\label{eq:pomdpx}
\mathbb{P}\left(x_{t+1} = j \mid x_t = i, u_t = a\right) &= p_{ij}^{a}, \\
\label{eq:pomdpy}
\mathbb{P}\left(y_{t+1} = \theta \mid x_{t+1} = j, u_t = a\right) &= r_{j\theta}^{a},
\end{align}
where $t\in\mathbb{N}_0$, $i,j\in X$, $a\in U$, $\theta\in Y$. The state transition dynamics~\eqref{eq:pomdpx} correspond to a conventional Markov Decision Process (MDP, e.g.~\cite{puterman2014markov}). However, the control actions $u_t$ are to chosen based on the known initial state distribution $\pi_0 = \operatorname{pdf}(x_0)$ and the sequences of observations, $\{y_1,\ldots,y_t\}$, and controls $\{u_0,\ldots,u_{t-1}\}$, respectively. That is, we are choosing our control actions in a Hidden Markov Model (HMM, e.g.~\cite{elliott2008hidden}) setup. Notice that, while POMDPs conventionally do not have an initial observation $y_0$ in~\eqref{eq:pomdpy}, as is commonly assumed in nonlinear system models of the form~(\ref{eq:state}-\ref{eq:output}), one can easily modify this basic setup without altering the discussion below.

Given control action $u_t = a$ and measured output $y_{t+1} = \theta$, the information state $\pi_t$ in a POMDP is updated via
\begin{align*}
\pi_{t+1,j} = \frac{\sum_{i\in X} 
\pi_{t,j} p_{ij}^a r_{j\theta}^a}{\sum_{i,j\in X} \pi_{t,j} p_{ij}^a r_{j\theta}^a} ,
\end{align*}
where $\pi_{t,j}$ denotes the $j^{\text{th}}$ entry of the row vector $\pi_t$. To specify the cost functionals~\eqref{eq:scostih} and~\eqref{eq:scostfh} in the POMDP setup, we write the stage cost as $c(x_t,u_t) = c_{i}^{a}$ if $x_t = i\in X$ and $u_t = a\in U$, summarized in the column vectors $c(a)$ of the same dimension as row vectors $\pi_k$. Similarly, the terminal cost terms are $c_N(x_t) = c_{i,N}$ if $x_N = i\in X$, summarized in the column vector $c_N$. The infinite horizon cost functional defined in Section~\ref{sec:smpc} then follows as 
\begin{align*}
J_\infty(\pi_{0},g) & =
\mathbb{E}_0 \left[ \sum_{k=0}^\infty{\alpha^k \pi_k c(g_k(\pi_k)) } \right],
\end{align*}
with corresponding finite-horizon variant
\begin{align*}
J_N(\pi_{0},g) &= 
\mathbb{E}_0 \left[ \sum_{k=0}^{N-1}{\alpha^k \pi_k c(g_k(\pi_k))} + 
\alpha^{N}\pi_N c_N \right].
\end{align*}
Extending~(\ref{eq:DP1}-\ref{eq:DP2}), optimal control decisions may then be computed via
\begin{multline}\label{eq:DP1pomdp}
J_{N-k}^{\star}(\pi_k) = \min_{g_k(\cdot)} \Bigg\{ \pi_{k} c(g_k(\pi_k)) \\ +
\alpha\sum_{\theta \in {Y}} \mathbb{P}\left(y_{k+1} = \theta \mid \pi_k,\,g_k(\pi_k) \right) 
J_{N-k-1}^{\star}(\pi_{k+1}) \Bigg\}, 
\end{multline}
for $k = 0,\ldots,N-1$, from terminal value function
\begin{align}\label{eq:DP2pomdp}
J_{0}^{\star}(\pi_N) = \pi_N c_N .
\end{align}
Using the notation for optimal finite- and infinite-horizon value functions as well as MPC policies introduced in Section~\ref{sec:smpc}, we next prove the following auxiliary result before extending the performance guarantees in Theorem~\ref{thm:smpc} to SMPC on POMDPs. 
\begin{lem}\label{lem:pomdp}
If there exist $\gamma\in[0,1]$ and $\eta\in\mathbb{R}_+$ such that
\begin{multline}\label{eq:lem}
\mathbb{E}_0 \left[ J_{N}^{\star}(\pi_1) - J_{N-1}^{\star}(\pi_1) \right] \leq 
\gamma \mathbb{E}_0 \left[ \pi_0 c(\mu_0^N(\pi_0)) \right] + \eta,
\end{multline}
for all densities $\pi_0$ of $x_0\in X$, then SMPC with discount factor $\alpha\in[0,1)$ yields
\begin{multline}\label{eq:bounds}
(1- \alpha\gamma)\, J^\star_{\infty}(\pi_0) \leq 
(1- \alpha\gamma)\, J_{\infty}(\pi_0,\mu^N_{\text{MPC}}) \\ \leq 
J^\star_{N}(\pi_0) + \frac{\alpha}{1 - \alpha}\eta.
\end{multline}
\end{lem} 
\begin{proof}
Optimality of the initial policy $\mu_0^{N}(\cdot)$ implies 
\begin{align*}
J_{N}^{\star}(\pi_0) &= \mathbb{E}_0 \left[ \pi_0 c(\mu_0^{N}) + 
\alpha J_{N-1}^{\star}(\pi_1) \right] \\ &+
\alpha \mathbb{E}_0 \left[  J_{N}^{\star}(\pi_1) - J_{N}^{\star}(\pi_1) \right],
\end{align*}
which by~\eqref{eq:lem} yields
\begin{multline}\label{eq:Vineq}
(1 - \alpha \gamma)\mathbb{E}_0
\left[ \pi_0 c(\mu_0^N(\pi_0)) \right] \leq\\ 
J_N^{\star}(\pi_0) -
\alpha  \mathbb{E}_0 \left[ J_{N}^{\star}(\pi_1) \right] + \alpha \eta.
\end{multline}
Now denote by $J_{\infty}^{M}(\pi_0,\mu^N_{\text{MPC}})$ the first $M \in \mathbb{N}_1$ terms of the achieved infinite-horizon cost $J_{\infty}(\pi_0,\mu^N_{\text{MPC}})$ subject to the SMPC implementation of policy $\mu_0^{N}(\cdot)$. By~\eqref{eq:Vineq}, we have
\begin{multline*}
(1 - \alpha \gamma) J_{\infty}^{M}(\pi_0,\mu^N_{\text{MPC}}) = \\
(1 - \alpha \gamma) \mathbb{E}_0 \left[ \sum_{k=0}^{M-1} \alpha^k \pi_k c(\mu_0^N(\pi_k)) \right] \leq \\
\mathbb{E}_0 \left[ J_N^{\star}(\pi_0) - \alpha J_N^{\star}(\pi_1) + \alpha\eta + \alpha J_N^{\star}(\pi_1) - \alpha^2 J_N^{\star}(\pi_2)+ \right. \\
\left. \alpha^2\eta +\ldots + 
\alpha^{M-1}J_N^{\star}(\pi_{M-1}) - \alpha^M J_N^{\star}(\pi_M) + \alpha^{M} \eta \right],
\end{multline*}
such that
\begin{multline*}
(1 - \alpha \gamma)  J_{\infty}^{M}(\pi_0,\mu^N_{\text{MPC}}) \leq 
J_N^{\star}(\pi_0)  \\
-\alpha^M \mathbb{E}_0 \left[ J_N^{\star}(\pi_M) \right] + 
\left( \alpha + \ldots +\alpha^{M}\right) \eta,
\end{multline*}
which confirms the right-hand inequality in~\eqref{eq:bounds} in the limit as $M\to\infty$. The left-hand inequality follows directly from optimality on the infinite horizon.
\end{proof}
This lemma then leads to the following assumption and subsequent performance result in the spirit of Theorems~\ref{thm:det}-\ref{thm:smpc}.
\begin{assum}\label{assm:pomdp}
For $\alpha\in[0,1)$, there exist $\eta\in\mathbb{R}_+$ and a policy $\tilde{g}(\cdot)$ such that
\begin{align}\label{eq:pomdpassm}
\mathbb{E}_{0}\left[\alpha\, \pi_1 c_N \right] \stackrel{}{\leq}
\mathbb{E}_{0}\left[ \pi_0 c_N - \pi_0 c(\tilde{g}(\pi_0))  \right]
 + \frac{\eta}{\alpha^{N-1}},
\end{align}
for all densities $\pi_0$ of $x_0\in X$.
\end{assum}
\begin{thm}\label{thm:pomdp}
[Performance of SMPC for POMDPs]\label{thm:pomdp}
\hspace{2cm}
Given Assumption~\ref{assm:pomdp}, SMPC for POMDPs with $\alpha\in[0,1)$ yields 
\begin{align*}
J_{\infty}^{\star}(\pi) \leq 
J_{\infty}(\pi,\mu^N_{\text{MPC}}) \leq 
J_{N}^{\star}(\pi) + \frac{\alpha}{1 - \alpha}\eta,
\end{align*}
for all densities $\pi$ of $x\in X$.
\end{thm} 
\begin{proof}
Use optimality and Assumption~\ref{assm:pomdp} to conclude
\begin{align*}
J_{N}^{\star}(&\pi_1) - J_{N-1}^{\star}(\pi_1) \\
=\ &\mathbb{E}_{0} \Bigg[ \left(\sum_{k=0}^{N-1}  \alpha^k 
\pi_{k+1}c(\mu_k^{N}(\pi_{k+1})) + \alpha^N \pi_{N+1}c_N \right)\\
& -\left(\sum_{k=0}^{N-2}  \alpha^k 
\pi_{k+1}c(\mu_{k+1}^{N}(\pi_{k+1})) + \alpha^{N-1} \pi_{N}c_N \right) \Bigg] \\ \leq \ & 
\mathbb{E}_{0} [ \alpha^{N-1}\pi_{N} c(\tilde{g}(\pi_{N})) \\
& + \alpha^N \pi_{N+1} c_N - \alpha^{N-1}\pi_{N} c_N ] \\ 
\leq \ &\eta,
\end{align*}
which implies~\eqref{eq:lem} with $\gamma = 0$ and thus completes the proof by Lemma~\ref{lem:pomdp}.
\end{proof}

\section{Numerical Example in Healthcare}
\label{sec:eg}
\subsection{Problem Setup}
The remainder of this paper discusses a particular numerical example of decisions on treatment and diagnosis in healthcare, displaying specifically the use of dual control in SMPC applied to a POMDP. Consider a patient treated for a specific disease which can be managed but not cured. For simplicity, we assume that the patient does not die under treatment. While this transition would have to be added in practice, it results in a time-varying model, which we avoid in order to keep the following discussion compact. 

The example, introduced in~\cite{sehr2017tractable}, is set up as follows. The disease encompasses three stages with severity increasing from Stage 1 through Stage 2 to Stage 3, transitions between which are governed by a Markov chain with transition probability matrix
\begin{align*}
P = \begin{bmatrix}
0.8&0.2&0.0\\0.0&0.9&0.1\\0.0&0.0&1.0
\end{bmatrix},
\end{align*}
where $P$ is the matrix with values $p_{ij}$ at row $i$ and column $j$. All transition and observation probability matrices below are defined similarly. Once our patient enters Stage 3, Stages 1 and 2 are inaccessible for all future times. However, Stage 3 can only be entered through Stage 2, a transition from which to Stage 1 is possible only under costly treatment. The same treatment inhibits transitions from Stage 2 to Stage 3. We have access to the patient state only through tests, which will result in one of three possible values, each of which is representative of one of the three disease stages. However, these tests are imperfect, with non-zero probability of returning an incorrect disease stage. All possible state transitions and observations are illustrated in Figure~\ref{fig:transitions}.

\begin{figure}[tb]
  \centering
  \includegraphics[width=\columnwidth]{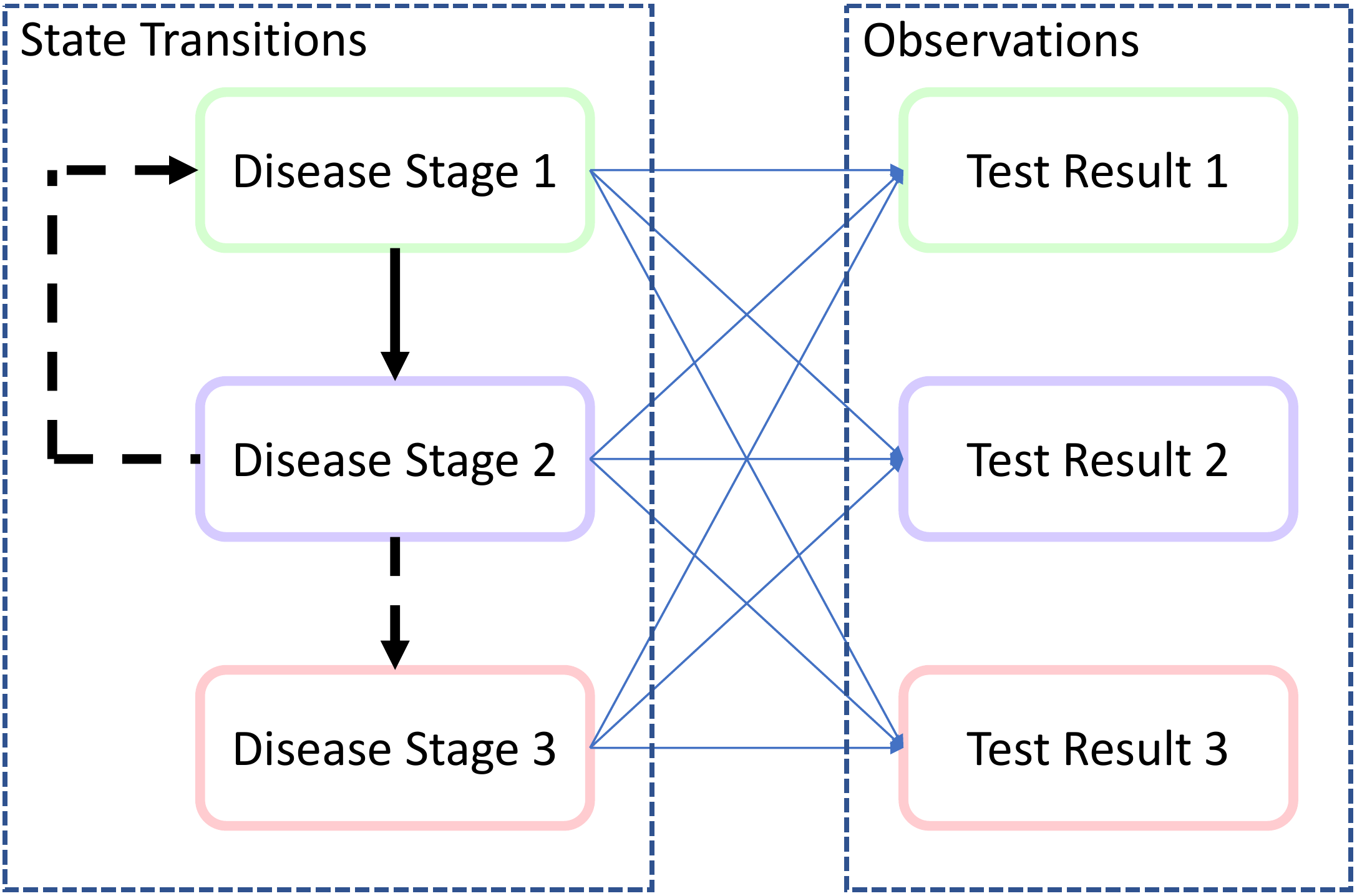}
  \caption{Feasible state transitions and possible test results in healthcare example. Solid arrows for feasible state transitions and observations. Dashed arrows for transitions conditional on treatment and diagnosis decisions.}
  \label{fig:transitions}
\end{figure}

At each point in time, the current information state $\pi_t$ is available to make one of four possible decisions:
\begin{enumerate}
\item Skip next appointment slot
\item Schedule new appointment
\item Order rapid diagnostic test
\item Apply available treatment
\end{enumerate}
Skipping an appointment slot results in the patient progressing through the Markov chain describing the transition probabilities of the disease without medical intervention, without new information being available after the current decision epoch. Scheduling an appointment does not alter the patient transition probabilities but provides a low-quality assessment of the current disease stage, which is used to refine the next information state. The third option, ordering a rapid diagnostic test, allows for a high-quality assessment of the patient's state, leading to a more reliable refinement of the next information state than possible when choosing the previous decision option. The results from this diagnostic test are considered available sufficiently fast so that the patient state remains unchanged under this decision. The remaining option entails medical intervention, allowing transition from Stage 2 to Stage 1 while preventing transition from Stage 2 to Stage 3. Transition probabilities $P(a)$, observation probabilities $R(a)$, and stage cost vectors $c(a)$ for each decision are summarized in Table~\ref{tab:numbers}. Additionally, we impose the terminal cost
\begin{align*}
c_N = \begin{bmatrix} 0 & 4 & 30 \end{bmatrix}^T.
\end{align*}

\begin{table*}[tb]
  \caption{Problem data for healthcare decision making example.}
  \centering
  \begin{tabular}{l|ccc}	
    Decision $a$ & Transition Probabilities $P(a)$ & Observation Probabilities $R(a)$ & Cost $c(a)$ \\
    \midrule
\hspace{0.2cm}1: Skip next appointment slot & 
$\begin{bmatrix} 0.80&0.20&0.00\\0.00&0.90&0.10\\0.00&0.00&1.00 \end{bmatrix}$ & 
$\begin{bmatrix} 1/3&1/3&1/3\\1/3&1/3&1/3\\1/3&1/3&1/3 \end{bmatrix}$ & 
$\begin{bmatrix} 0\\5\\5 \end{bmatrix}$ \\[0.5cm]
\hspace{0.2cm}2: Schedule new appointment & 
$\begin{bmatrix} 0.80&0.20&0.00\\0.00&0.90&0.10\\0.00&0.00&1.00 \end{bmatrix}$ & 
$\begin{bmatrix} 0.40&0.30&0.30\\0.30&0.40&0.30\\0.30&0.30&0.40 \end{bmatrix}$ & 
$\begin{bmatrix} 1\\1\\1 \end{bmatrix}$ \\[0.5cm]
\hspace{0.2cm}3: Order rapid diagnostic test & 
$\begin{bmatrix} 1.00&0.00&0.00\\0.00&1.00&0.00\\0.00&0.00&1.00 \end{bmatrix}$ & 
$\begin{bmatrix} 0.90&0.05&0.05\\0.05&0.90&0.05\\0.05&0.05&0.90 \end{bmatrix}$ & 
$\begin{bmatrix} 4\\3\\4 \end{bmatrix}$ \\[0.5cm]
\hspace{0.2cm}4: Apply available treatment & 
$\begin{bmatrix} 0.80&0.20&0.00\\0.75&0.25&0.00\\0.00&0.00&1.00 \end{bmatrix}$ & 
$\begin{bmatrix} 0.40&0.30&0.30\\0.30&0.40&0.30\\0.30&0.30&0.40 \end{bmatrix}$ & 
$\begin{bmatrix} 4\\2\\4 \end{bmatrix}$
  \end{tabular}
  \label{tab:numbers}
\end{table*}

\subsection{Rationale for Duality}
Intuitively, we expect an efficient policy for this problem to attempt avoiding transitions to Stage 3 while managing the resources required to schedule appointments, order tests, or apply medical intervention. This may, in principle, be achieved by a policy akin to the following structure:
\begin{enumerate}
\item Skip appointments when Stages 2 and 3 are unlikely.
\item Schedule appointments when Stages 2 and 3 are likely but the probability for Stage 2 is below some threshold.
\item Order diagnostic test if the probability of Stage 2 lies in a specific range.
\item Proceed with medical intervention if the probability of Stage 2 is high.
\end{enumerate}
While the optimal policy may be somewhat more intricate, this simple decision structure could be acceptable in practice. However, even this simple structure includes duality in the decisions, demonstrated by including the diagnostic test even though it does not alter the patient state. That is, this decision improves the quality of available information at a cost, also called \emph{exploration}. This improvement in the available information allows us to apply medical intervention at appropriate times, which is called \emph{exploitation}. 

\subsection{Computational Results}
The trade-off between these two principal decision categories is precisely what is encompassed by duality, which we can include in an optimal sense by solving~(\ref{eq:DP1pomdp}-\ref{eq:DP2pomdp}) and applying the resulting initial policy in receding horizon fashion. This is demonstrated in Figure~\ref{fig:simN5}, which shows simulation results for SMPC with control horizon $N = 5$ and discount factor $\alpha = 0.98$. As anticipated, the stochastic optimal receding horizon policy shows a structure not drastically different from the decision structure motivated above. In particular, diagnostic tests are used effectively to decide on medical intervention. 

In order to apply Theorem~\ref{thm:pomdp} to this particular example, we choose the policy $\tilde{g}(\cdot)$ in Assumption~\ref{assm:pomdp} always to apply medical intervention. Using the worst-case scenario for the expectations in~\eqref{eq:pomdpassm}, which entails transition from Stage 1 to Stage 2 under treatment, we can satisfy Assumption~\ref{assm:pomdp} with $\eta = 7.92$. The computed cost in our simulation is $J_{N}^{\star}(\pi_0) \approx 11.36$. Combined with the discount factor $\alpha = 0.98$, we thus have the upper bound
\begin{align*}
J_{\infty}(\pi_0,\mu_{\text{MPC}}^N) &\leq J_{N}^{\star}(\pi_0) + \frac{\alpha}{1-\alpha}\eta \approx
400
\end{align*}
via application of Theorem~\ref{thm:pomdp}. Denoting by $e_j$ the row-vector with entry $1$ in element $j$ and zeros elsewhere, the observed (finite-horizon) cost corresponding with Figure~\ref{fig:simN5} is
\begin{align*}
J_{\infty}^{\text{obs}} = \sum_{k = 0}^{29}e_{x_k}c(\mu_0^N(\pi_k)) \approx 38.53 < 400.
\end{align*}
While this bound is not particularly tight, one may modify the discount factor $\alpha$ or the terminal cost $c_N$ to achieve a tighter estimate of the achieved MPC cost.

\begin{figure*}[h]
  \centering
  \includegraphics[width=\linewidth]{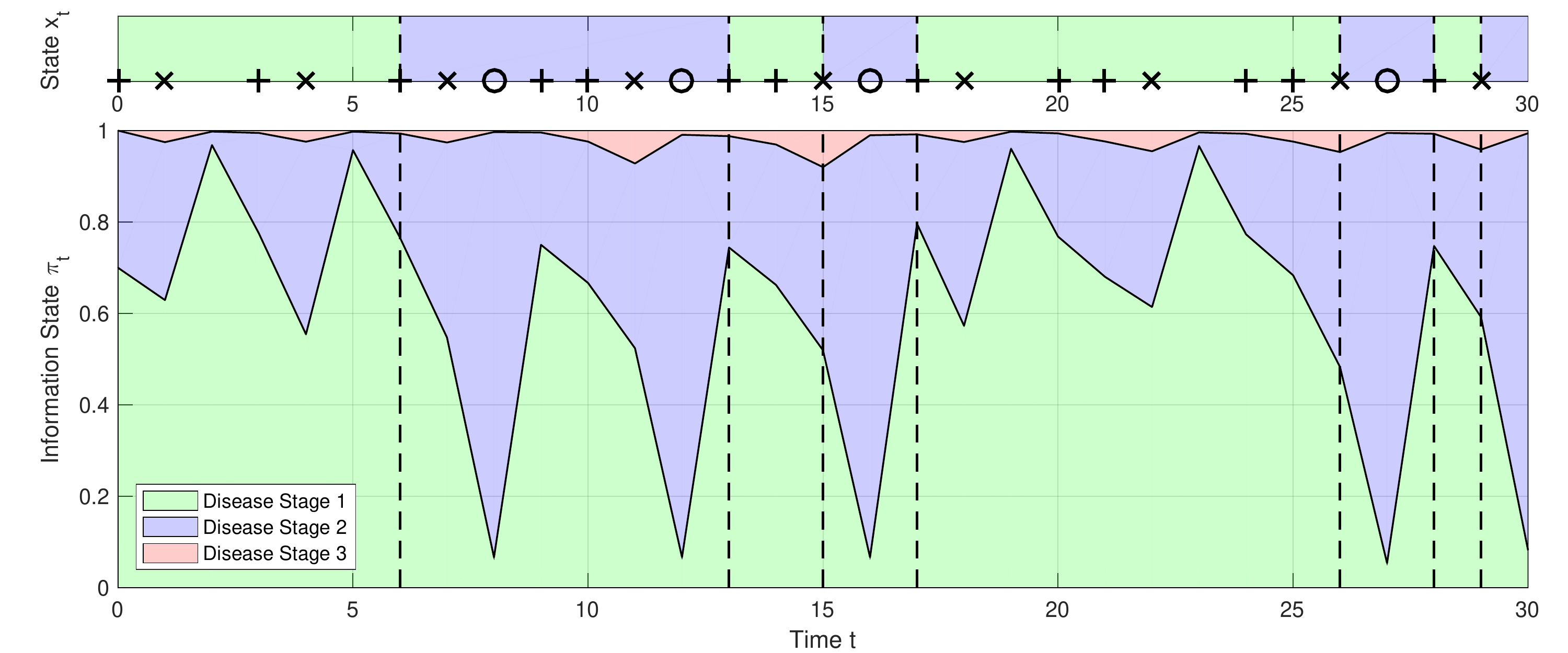}
  \caption{Simulation results for SMPC with horizon $N=5$ and discount factor $\alpha = 0.98$. Top plot displays patient state and transitions, with optimal SMPC decisions based on current information state: appointment (pluses); diagnosis (crosses); treatment (circles). Bottom plot shows information state evolution. Dashed vertical lines mark time instances of state transitions.}
  \label{fig:simN5}
\end{figure*}

\section{Conclusions}
\label{sec:conclusions}
We extended closed-loop achieved performance guarantees well-known in deterministic MPC to SMPC and in particular receding horizon control of POMDPs, which allow tractable solution of the underlying Stochastic Optimal Control problems and thus duality of the control inputs in an optimal sense. The basic formulations in this paper can be modified, for instance, by introducing state and input constraint sets or time-varying (monotonic) stage costs. While this requires additional assumptions to maintain recursive feasibility of the MPC and SMPC inputs, the cost discussion is rather similar. We demonstrated use of the novel results using a particular POMDP instance in healthcare decision making, demanding the use of probing control inputs in order to adequately decide upon the proper and cost-effective use of medical intervention.

\bibliographystyle{ieeetr}
\bibliography{CDC2017}

\end{document}